\documentclass[10pt]{amsart}
 \usepackage{amssymb}
 \usepackage{amssymb}
 \usepackage{amsmath}
 \usepackage{graphicx}
 \usepackage{epstopdf}
\newtheorem{theorem}{Theorem}[section]

\newtheorem{lemma}[theorem]{Lemma}

\begin{document}

\title[The Double Dispersion Equation]{Instability and stability properties of traveling waves for the double dispersion equation}

\author[H. A. Erbay]{H. A. Erbay}
\address{Department of Natural and Mathematical Sciences, Faculty of Engineering, Ozyegin University,  Cekmekoy 34794, Istanbul, Turkey  }
\email{husnuata.erbay@ozyegin.edu.tr}

\author[S. Erbay]{S. Erbay}
\address{Department of Natural and Mathematical Sciences, Faculty of Engineering, Ozyegin University,  Cekmekoy 34794, Istanbul, Turkey  }
\email{saadet.erbay@ozyegin.edu.tr}

\author[A. Erkip]{A. Erkip}
\address{Faculty of Engineering and Natural Sciences, Sabanci University,  Tuzla 34956,  Istanbul,    Turkey}
\email{albert@sabanciuniv.edu}

\date{\today}
\maketitle

  \begin{abstract}
 In this article we are concerned with the instability and stability properties of traveling wave solutions of the double dispersion equation $~u_{tt} -u_{xx}+a u_{xxxx}-bu_{xxtt}  = - (|u|^{p-1}u)_{xx}~$ for $~p>1$, $~a> b>0$. The main characteristic of this equation is the existence of two sources of dispersion, characterized by the terms $u_{xxxx}$ and  $u_{xxtt}$.     We obtain an explicit condition in terms of $a$, $b$ and $p$ on wave velocities ensuring that   traveling wave solutions of the double dispersion equation  are strongly unstable by blow up. In the special case of the Boussinesq equation ($b=0$), our condition reduces to the one given in the literature. For the double dispersion equation, we also investigate orbital stability of traveling waves  by considering the convexity of a scalar function. We provide  analytical as well as numerical results on the variation of the stability region of wave velocities with  $a$, $b$ and $p$ and then state explicitly the conditions under which the traveling waves are orbitally stable.
 \end{abstract}

\setcounter{equation}{0}
\section{Introduction}
\noindent

The present paper is  concerned with the instability and stability properties of traveling wave solutions for the   double dispersion equation
\begin{equation}
    u_{tt} -u_{xx}+a u_{xxxx}-bu_{xxtt}  = -(|u|^{p-1}u)_{xx},   \label{dde0}
\end{equation}
where $a$, $b$ are positive real constants with $a > b$, and $p>1$. In particular we prove that traveling wave solutions are unstable by blow-up if the wave velocities of  the traveling waves are less than  a critical wave velocity. We also state explicitly a set of conditions on  $a$, $b$ and $p$  for which the traveling waves  are orbitally stable.

The double dispersion equation (\ref{dde0}) was derived as a mathematical model of the propagation of dispersive waves in a wide variety of situations, see for instance \cite{samsonov, porubov} and the references therein.  Well posedness (and related properties) of the Cauchy problem  for the double dispersion equation  have been studied in the literature  by several authors \cite{wang, yacheng, kutev}. It is interesting to note that (\ref{dde0}) is a special case of the general class of nonlinear nonlocal wave equations
\begin{equation}
    u_{tt}-Lu_{xx}=B(g(u))_{xx},  \label{class0}
\end{equation}
with pseudo-differential operators $L$ and $B$, studied in \cite{babaoglu, erbay-NA, erbay-JMAA}.  Indeed, for  the case
\begin{equation}
    L=( I-aD_{x}^{2}) ( I-bD_{x}^{2})^{-1}, ~~~~ B=( I-bD_{x}^{2})^{-1}, ~~~g(u)=- |u|^{p-1}u, \label{reduction}
\end{equation}
where $I$ is the identity operator and $D_x$ denotes the partial derivative with respective to $x$,  (\ref{class0}) reduces to   (\ref{dde0}).  The well-posedness of the Cauchy problem for the general class (\ref{class0})  was studied in \cite{babaoglu} and then the parameter dependent thresholds for global existence versus  blow-up were established in \cite{erbay-NA} for power nonlinearities. In a recent study \cite{erbay-JMAA}  on (\ref{class0}), again for power nonlinearities, the existence of traveling wave solutions $u=\phi_{c}(x-ct)$, where $c\in \mathbb{R}$ is the wave velocity,  has been established and  orbital stability of the traveling waves has been studied. The orbital stability is based on the convexity of a certain function $d(c) $ related to conserved quantities. Furthermore, it has been shown that when $L=I$, (\ref{class0}) becomes a special case of the Klein-Gordon-type equations and $d(c)$ can be computed explicitly.     In \cite{erbay-JMAA} the sharp threshold of instability/stability of traveling waves for this regularized Klein-Gordon equation has been established. In other words, for $L=I$, it has been shown that  traveling wave solutions of (\ref{class0})  are orbitally stable  for
\begin{equation}
  \frac{p-1}{p+3} <  c^{2} < 1  \label{tresh1}
\end{equation}
and are unstable by blow-up for
\begin{equation}
 c^{2} < \frac{p-1}{p+3}.   \label{tresh2}
\end{equation}
It remains an open question, however, whether a sharp threshold of instability/stability  can be obtained  for the double dispersion equation (\ref{dde0}) which is another  special case of (\ref{class0}).

For some limiting cases of (\ref{dde0}), the above question was fully answered in the literature. For the special case  $a=1$, $b=0$; (\ref{dde0}) becomes the  (generalized) Boussinesq equation \cite{bous}
\begin{equation}
    u_{tt} - u_{xx} + u_{xxxx} = -(|u|^{p-1}u)_{xx} \label{gen-bouss}
\end{equation}
which has received much attention in the literature. It was established in \cite{bona-sachs} that solitary wave solutions of (\ref{gen-bouss}) are orbitally stable if
\begin{equation}
     \frac{p-1}{4}<c^{2}<1  ~~~~\mbox{and} ~~~~ 1<p<5.  \label{bona}
\end{equation}
In \cite{liu1993}, it was  proved  that solitary waves for (\ref{gen-bouss}) are orbitally unstable if
\begin{equation}
     c^{2}<\frac{p-1}{4}  ~~~~\mbox{and} ~~~~ 1<p<5,
\end{equation}
or
\begin{equation}
     c^{2}<1  ~~~~\mbox{and} ~~~~ p\geq 5.
\end{equation}
On the other hand, in \cite{todorova} it was shown that traveling wave solutions of (\ref{gen-bouss}) are strongly unstable by blow-up  for
\begin{equation}
    c^{2}<  \frac{p-1}{2(p+1)}. \label{ohta}
\end{equation}
In the limiting case $a=b$; (\ref{dde0}) reduces to
\begin{equation}
    u_{tt} -u_{xx} = -( 1-bD_{x}^{2})^{-1}(|u|^{p-1}u)_{xx},   \label{klein}
\end{equation}
which is a special case of the regularized Klein-Gordon equation studied in \cite{erbay-JMAA} and therefore the results given by (\ref{tresh1}) and (\ref{tresh2}) are also valid for this special case. For the special case $a=0$, $b=1$; (\ref{dde0}) reduces to the improved Boussinesq equation \cite{ost}
\begin{equation}
    u_{tt} - u_{xx} - u_{xxtt} = -(|u|^{p-1}u)_{xx}, \label{imp-boussinesq}
\end{equation}
which has no traveling wave solution due to the minus sign on the right hand side.  For the sake of completeness, we point out that, in \cite{stubbe}, a sufficient condition for  orbital stability of solitary waves was given for a more general version of (\ref{dde0}):
\begin{equation}
    \left( 1+\gamma \left\vert D_{x}\right\vert^{\nu }\right) u_{tt}-\left(a_{0}+a_{1}\left\vert D_{x}\right\vert ^{\nu }\right) u_{xx}
            =-\left( \left\vert u\right\vert^{p -1}u\right) _{xx},
\end{equation}
where $\nu \geq 1$, $\gamma>0$, $a_{0}$ and $a_{1}$ are real constants.

The aim of the present study is to investigate  instability/stability properties of traveling wave solutions for (\ref{dde0}) when $~a> b>0$.  Our main result  is that for all wave velocities $c$ with $c^{2}< c_{0}^{2}$ where
\begin{equation}
     c_{0}^{2}=  \left(\frac{p-1}{p+1}\right) \left[ 1+\left( 1-\frac{b(p+3)(p-1)}{a(p+1)^2} \right)^{1/2}\right]^{-1},  \label{thresh3}
\end{equation}
traveling wave solutions of (\ref{dde0}) are unstable by blow-up. It is important to note that our condition $c^{2}< c_{0}^{2}$ for instability by blow-up matches the known results in the two limiting cases $a=1$, $b=0$ and $a=b$. That is, as it is expected, it reduces to (\ref{ohta})  when $a=1$, $b=0$ and to (\ref{tresh2}) when $a=b$. For the other result of this work,    we investigate both analytically and numerically orbital stability of traveling waves by applying the  convexity criterion  to (\ref{dde0}). We then identify conditions (see (\ref{conA})-(\ref{conC})) on wave velocity and the parameters $a$, $b$ and $p$ for which traveling wave solutions of (\ref{dde0})  are orbitally stable.    Recalling that we restrict the discussion to the case $~a> b>0$, one may ask whether similar conclusions are still true if $a<b$. We emphasize that for $a<b$, (\ref{dde0}) has traveling wave solutions with $c^{2}< a/b$ but we cannot make a conclusion about instability by blow-up in this case. The crucial fact is that for $a<b$ the dispersive term $u_{xxtt}$ in (\ref{dde0}) dominates and thus (\ref{dde0}) behaves much like (\ref{imp-boussinesq}). It seems that  our restriction $a > b$ is more structural than a technical one.

The structure of  the paper is as follows. In Section 2 we first review some previously known results, including the local existence theorem and the conserved quantities, for (\ref{dde0}) and then discuss the Pohozaev identities and the invariant sets.  In Section 3,  we prove instability by blow-up of traveling waves with $c^{2}< c_{0}^{2}$ for (\ref{dde0}). In Section 4, we announce orbital stability conditions for   traveling wave solutions of (\ref{dde0}).

Throughout this paper, we use the standard notation  for  function spaces. The symbol $\widehat u$ represents the Fourier transform of $u$, defined by $\widehat u(\xi)=\int_\mathbb{R} u(x) e^{-i\xi x}dx$. The $L^p$ ($1\leq p <\infty$)  norm of $u$ on $\mathbb{R}$ is denoted by  $\Vert u\Vert_{L^p}$. The inner product of $u$ and $v$ in $L^2$ is represented by $\langle u, v\rangle$. The $L^{2}$ Sobolev space of order $s$ on $\mathbb{R}$ is denoted by  $H^{s}=H^s(\mathbb{R})$ with the norm $\Vert u\Vert_{H^{s}}^2=\int_\mathbb{R} (1+\xi^2)^s |\widehat u(\xi)|^2 d\xi$.  The symbol $\mathbb{R}$ in $\int_{\mathbb{R}}$ will be mostly suppressed to simplify exposition.

\setcounter{equation}{0}
\section{Pohozaev Identities and Invariant Sets}
\noindent

\subsection{Preliminaries: Local Existence and Conserved Quantities}

We now list some preliminary results for (\ref{dde0}) (or equivalently, for (\ref{class0}) with   (\ref{reduction})). Local existence of the Cauchy problem for (\ref{dde0}) has been established in \cite{wang}. The local existence result given in \cite{babaoglu} for (\ref{class0}) will also apply. For our purposes it is sufficient to consider solutions in $ H^{1}$ and therefore we restrict our remarks concerning (\ref{dde0}) to this case. The local existence result in \cite{wang} implies that for initial data  in $H^{1}\times L^{2}$, the Cauchy problem for (\ref{dde0}) has a unique solution $u\in C([0,T),H^{1})\cap C^{1}([0,T),L^{2})$ for some $T>0$.  As in \cite{erbay-NA}, we now introduce new variables $(u, w) $, where $u=v_{x}$ and $w=v_{t}$  for a suitable function $v$. Then we consider the following equivalent initial-value problem:
\begin{eqnarray}
    && u_t=w_x,~~~x\in \mathbb{R},~~t>0  \label{sys1} \\
    && w_t= ( 1-b D_{x}^{2})^{-1}\left[( 1-a D_{x}^{2}) u_x-(|u|^{p-1}u)_{x}\right],~~~x\in \mathbb{R},~~t>0  \label{sys2} \\
    && u(x,0)=u_{0}(x),~~~w(x,0)=w_{0}(x),~~x\in \mathbb{R} \label{sys3}
\end{eqnarray}
for which the local existence theorem  in \cite{erbay-NA} is rephrased  as follows.
\begin{theorem}\label{theo2.1}
     For initial data $U_{0}=(u_{0},w_{0})\in H^{1}\times H^{1}$, there exists some $T>0$ so that the Cauchy problem (\ref{sys1})-(\ref{sys3}) is locally well-posed with solution $U=(u,w)\in C([0,T), H^{1}\times H^{1})$.
\end{theorem}

The energy and momentum functionals given in \cite{erbay-NA} turn out to be
\begin{eqnarray}
    \!\!\!\!\!\!\!\!\!\!\! \!\!\!\!\!\!\!\!\!\!\!
    && E(U) = E(u,w)=\frac{1}{2}\int (w^{2}+b w_{x}^{2})dx+\frac{1}{2}\int (u^{2}+a u_{x}^{2})dx-\frac{1}{p+1}\int |u|^{p+1}dx,   \label{energy}\\
    \!\!\!\!\!\!\!\!\!\!\! \!\!\!\!\!\!\!\!\!\!\!
    && M(U) = M(u,w)=\int (uw+bu_{x}w_{x})dx \label{momentum}
\end{eqnarray}%
for (\ref{sys1})-(\ref{sys2}).  The energy and momentum are conserved quantities of (\ref{sys1})-(\ref{sys2}), namely for a solution $U(t)$ of (\ref{sys1})-(\ref{sys2}) both $E(U(t))$ and $M(U(t))$ are independent of $t$ \cite{erbay-NA}. We note that $H^{1}\times H^{1}$ is the natural energy and momentum space.

\subsection{Pohozaev Identities and Invariant Sets}
\noindent

Traveling wave solutions $u(x, t) =\phi_{c} (x-ct)$ of (\ref{dde0})  satisfy the differential equation
\begin{equation}
    ( a-bc^{2}) \phi_{c}^{\prime \prime }-(1- c^{2}) \phi_{c} +\vert \phi_{c} \vert^{p-1}\phi_{c} =0  \label{phi}
\end{equation}%
where we have assumed that $\phi_{c}$ and all its derivatives  decay at infinity. For $a-bc^{2}>0$ and $1-c^{2}>0$,  (\ref{phi}) has a unique solution up to translation, namely
\begin{equation}
    \phi_{c}(x) =\left[ \frac{1}{2}(p+1)(1-c^{2})\right]^{\frac{1}{p-1}}\mbox{sech}^{\frac{2}{p-1}}\left[ \frac{1}{2}(p-1)( \frac{1-c^{2}}{a-bc^{2}})^{\frac{1}{2}}x \right]. \label{solitary}
\end{equation}%
As we assume that $a> b$, the above two conditions  given for the wave velocity reduce to $c^{2}<\min \left\{1,  a/b\right\} =1$. We note that this is exactly the bound obtained in \cite{erbay-JMAA}, which is due to the fact that  the symbol $l(\xi)$  of the operator $L$ in (\ref{reduction}) satisfies
\begin{displaymath}
    1\leq l(\xi) =\frac{1+a\xi^{2}}{1+b\xi^{2}} \leq \frac{a}{b}
\end{displaymath}
for $a > b$.

We make extensive use of the following two Pohozaev identities.
\begin{lemma}\label{lem2.2}
    Traveling wave solutions of (\ref{phi}) satisfy the Pohozaev identities
    \begin{eqnarray}
        (1-c^{2}) \Vert \phi_{c}\Vert_{L^{2}}^{2}+(a-bc^{2}) \left\Vert \phi_{c}^{\prime}\right\Vert_{L^{2}}^{2}
            -\left\Vert \phi_{c}\right\Vert_{L^{p+1}}^{p+1} &=& 0 \label{Ph1}\\
        \frac{(1-c^{2})}{2}\Vert \phi_{c}\Vert _{L^{2}}^{2}-\frac{(a-bc^{2})}{2}\Vert \phi_{c}^{\prime }\Vert_{L^{2}}^{2}
            -\frac{1}{p+1}\Vert \phi_{c}\Vert _{L^{p+1}}^{p+1} &=&0. \label{Ph2}
\end{eqnarray}
\end{lemma}
\begin{proof}
    The first identity is obtained  multiplying (\ref{phi}) by $\phi _{c}$ and then  integrating the resulting equation over $\mathbb{R}$. To obtain the second one we multiply (\ref{phi}) by $x \phi_{c}^{\prime }$ and again integrate.
\end{proof}

To simplify the notation, from now on we will fix $c$ with $c^{2}<1$ and let
\begin{displaymath}
A=1-c^{2},\quad B=a-bc^{2}.
\end{displaymath}
For $u\in H^{1}$ we define two functionals, $P_{1}$ and $P_{2}$, as follows:
\begin{eqnarray}
    P_{1}(u) &=& A\left\Vert u\right\Vert_{L^{2}}^{2}+B\left\Vert u_{x}\right\Vert_{L^{2}}^{2}-\left\Vert u\right\Vert_{L^{p+1}}^{p+1}, \label{p1} \\
    P_{2}(u) &=& \frac{A}{2}\left\Vert u\right\Vert_{L^{2}}^{2}-\frac{B}{2}\left\Vert u_{x}\right\Vert_{L^{2}}^{2}
        -\frac{1}{p+1}\left\Vert u\right\Vert _{L^{p+1}}^{p+1}.   \label{p2}
\end{eqnarray}
From Lemma \ref{lem2.2} we have $P_{1}(\phi_{c})=0$ and $P_{2}(\phi_{c})=0$. Moreover, we note that $P_{1}(u)$ coincides with the functional $2\mathcal{I}_{c}(u) -\mathcal{Q}(u)$ of \cite{erbay-JMAA} (and with $2\mathcal{I}_{\gamma}(u) -\mathcal{Q}(u)$ of \cite{erbay-NA}). As in \cite{erbay-NA} and \cite{erbay-JMAA}, using (\ref{energy}) and (\ref{momentum}) we get the following identity:
\begin{equation}
    E(u, w) +cM(u,w) ={1\over 2}\Vert w+cu\Vert_{L^{2}}^{2}+{b\over 2}\Vert w_{x}+cu_{x}\Vert_{L^{2}}^{2}+V(u), \label{EM}
\end{equation}
where $V(u)$ is defined as
\begin{equation}
    V(u) =\frac{A}{2}\left\Vert u\right\Vert_{L^{2}}^{2}+\frac{B}{2}\left\Vert u_{x}\right\Vert_{L^{2}}^{2}
        -\frac{1}{p+1}\left\Vert u\right\Vert_{L^{p+1}}^{p+1}.                             \label{Vu}
\end{equation}
In what follows,  for the traveling wave solution $u(x, t) =\phi_{c}(x-ct)$,  the corresponding solution of (\ref{sys1})-(\ref{sys2}) will be denoted by $U(x,t)=\Phi_{c}( x-ct)$ in which  $\Phi_{c}(x) =(\phi_{c}(x) ,-c\phi_{c}(x) )$. From (\ref{EM}) and (\ref{Vu}) it follows that
\begin{equation}
    E(\Phi_{c}) +cM(\Phi_{c}) =V(\phi_{c}).  \label{V}
\end{equation}
We now rephrase Lemma 4.1 of \cite{erbay-NA} and Lemma 4.2 of \cite{erbay-JMAA} together as follows:
\begin{lemma}\label{lem2.3}
    $d(c) =\inf \left\{ V(u) : u \in H^{1},~~~~ u\not=0,~~~~ P_{1}(u) =0 \right\} $ is attained at the travelling wave $\phi_{c}$. Moreover
    \begin{displaymath}
     \!\!\!\!\!\!\!\!\!\!\! \!\!\!\!\!\!\!\!\!\!\!
        \inf \left\{ E(U) +cM(U) : U=(u, w)\in H^{1}\times H^{1},~~~~ u\not=0,~~~~ P_{1}(u) =0\right\} =d(c).
    \end{displaymath}
\end{lemma}

For $\alpha\in \mathbb{R}$, we now define a functional, $K_{\alpha}(u)$, as follows:
\begin{eqnarray}
    \!\!\!\!\!\!\!\!\!\!\!\!\!\!\!\!\!\!\!\!\!\!
    K_{\alpha}(u)
    &=&\alpha P_{1}(u) +P_{2}(u)  \nonumber \\
    \!\!\!\!\!\!\!\!\!\!\!\!\!\!\!\!\!\!\!\!\!\!
    &=&\frac{A}{2}(2\alpha +1) \left\Vert u\right\Vert_{L^{2}}^{2}
        +\frac{B}{2}(2\alpha -1) \left\Vert u_{x}\right\Vert_{L^{2}}^{2}
        -(\alpha +\frac{1}{p+1}) \left\Vert u\right\Vert_{L^{p+1}}^{p+1}.  \label{Ku}
\end{eqnarray}
Note that $K_{\alpha}(\phi_{c}) =0$  for all $\alpha$. Consider the family of minimization problems
\begin{displaymath}
    d_{\alpha }(c)=\inf \left\{ V(u) : u\in H^{1},~~~~ u\not=0,~~~~ K_{\alpha}(u) =0\right\} .
\end{displaymath}
Following the scaling idea in \cite{coz}, we prove:
\begin{lemma}\label{lem2.4}
    For every $\alpha >\frac{1}{2}$ we have $d_{\alpha }(c)=d(c)$.
\end{lemma}
\begin{proof}
    Since $K_{\alpha }(\phi_{c}) =0$ we have $d_{\alpha}(c)\leq V(\phi_{c}) =d(c) $. For the converse, we take some $u\not=0$ with $K_{\alpha }(u) =0$.  If $P_{1}(u) =0$, then by Lemma \ref{lem2.3} we have $V(u) \geq d(c) $. We now turn to the case $P_{1}(u) \not=0$. For $\lambda >0$ we let $u_{\lambda}(x, t) =\lambda ^{\alpha }u\left( \frac{x}{\lambda },t\right)$. Substituting $u_{\lambda}$ into (\ref{p1}) yields
    \begin{eqnarray*}
        P_{1}(u_{\lambda})
        &=& A\lambda^{2\alpha +1}\left\Vert u\right\Vert_{L^{2}}^{2}+B\lambda^{2\alpha -1}\left\Vert u_{x}\right\Vert_{L^{2}}^{2}
            -\lambda^{\alpha (p+1)+1}\left\Vert u\right\Vert_{L^{P+1}}^{p+1} \\
        &=&\lambda^{2\alpha -1}\left(A\lambda^{2}\left\Vert u\right\Vert_{L^{2}}^{2}+B\left\Vert u_{x}\right\Vert_{L^{2}}^{2}
            -\lambda^{\alpha (p-1)+2}\left\Vert u\right\Vert_{L^{P+1}}^{p+1}\right),
    \end{eqnarray*}
    from which it follows that $P_{1}(u_{\lambda})$ is positive for small $\lambda$ but negative for large $\lambda$. Hence there is some $\lambda_{0}$ for which $P_{1}(u_{\lambda_{0}})=0$. Thus, by Lemma \ref{lem2.3}, we have  $V(u_{\lambda_{0}}) \geq d(c)$. On the other hand, computation of  $V(u)$ at  $u_{\lambda}$ gives
    \begin{displaymath}
        V\left( u_{\lambda }\right)
        =\frac{A}{2}\lambda ^{2\alpha +1}\left\Vert u\right\Vert _{L^{2}}^{2}+\frac{B}{2}\lambda ^{2\alpha -1}\left\Vert u_{x}\right\Vert_{L^{2}}^{2}
            -\frac{1}{p+1}\lambda^{\alpha (p+1)+1}\left\Vert u\right\Vert_{L^{P+1}}^{p+1}.
    \end{displaymath}
    Differentiating this we get
    \begin{eqnarray*}
        \frac{dV\left( u_{\lambda }\right) }{d\lambda }
        &=&\frac{A}{2}(2\alpha +1) \lambda^{2\alpha}\left\Vert u\right\Vert_{L^{2}}^{2}
            +\frac{B}{2}(2\alpha -1) \lambda^{2\alpha -2}\left\Vert u_{x}\right\Vert_{L^{2}}^{2}  \\
        &~& ~
            -\frac{\alpha (p+1)+1}{p+1}\lambda^{\alpha (p+1) }\left\Vert u\right\Vert _{L^{P+1}}^{p+1} \\
        &=& \lambda^{2\alpha -2} g(\lambda)
    \end{eqnarray*}
    with
    \begin{displaymath}
    \!\!\!\!\!\!\!\!
        g(\lambda) =\frac{A}{2}(2\alpha +1) \lambda^{2}\left\Vert u\right\Vert_{L^{2}}^{2}
            +\frac{B}{2}(2\alpha-1)\left\Vert u_{x}\right\Vert_{L^{2}}^{2}
            -\frac{\alpha (p+1) +1}{p+1}\lambda^{\alpha (p-1) +2}\left\Vert u\right\Vert_{L^{P+1}}^{p+1}.
    \end{displaymath}
    It is easy to se that $g^{\prime }(\lambda)$ changes sign from positive to negative exactly once on $(0,\infty )$. We observe that when $2\alpha -1>0$,   the function $g(\lambda)$ is positive for small $\lambda$ but negative for large $\lambda$. Hence we conclude that  $g(\lambda)$ changes its sign exactly once on $(0, \infty )$. The same conclusion holds  for $\frac{d}{d\lambda }V( u_{\lambda })$. This in turn shows that $V(u_{\lambda})$ attains its global maximum at exactly one point in $(0,\infty )$. Moreover
    \begin{eqnarray*}
        \!\!\!\!\!\!\!\!\!\!\!\!\!
        \frac{dV(u_{\lambda})}{d\lambda} \mid_{\lambda =1}
        &=&\frac{A}{2}(2\alpha +1) \left\Vert u\right\Vert_{L^{2}}^{2}
            +\frac{B}{2}(2\alpha -1)  \left\Vert u_{x}\right\Vert_{L^{2}}^{2}
            -\frac{\alpha (p+1) +1}{p+1}\left\Vert u\right\Vert_{L^{P+1}}^{p+1} \\
        \!\!\!\!\!\!\!\!\!\!\! \!\!\!\!\!\!\!\!\!\!\!
        &=& K_{\alpha }(u) =0,
    \end{eqnarray*}%
    so that the maximum is attained at $\lambda =1$. This means $V(u) \geq V(u_{\lambda_{0}}) \geq d(c)$. So we have $d_{\alpha}(c)\geq d(c)$. This completes the proof.
\end{proof}

We now let
\begin{displaymath}
    \widetilde{\Sigma}_{\alpha }=\left\{ U\in H^{1}\times H^{1}:E(U) +cM(U) <d(c),~~~~ K_{\alpha}(u) <0 \right\}.
\end{displaymath}
\begin{lemma}\label{lem2.5}
    Let $\ \alpha >{1\over 2}$. Then $\widetilde{\Sigma}_{\alpha}$ is invariant under the flow defined by the Cauchy problem (\ref{sys1})-(\ref{sys3}).
\end{lemma}
\begin{proof}
    Suppose $U_{0}\in \widetilde{\Sigma }_{\alpha }$ and let $U(t)$ be the solution of (\ref{sys1})-(\ref{sys3}) with initial value $U_{0}$. Since $E$ and $M$ are conserved quantities, then $E(U(t)) +cM(U(t)) <d(c)$. Assume that $U(t)$ does not stay in $\widetilde{\Sigma}_{\alpha}$. Then there is some $t_{1}$ for which  $K_{\alpha}(u(t_{1})) =0$. Thus, by Lemma  \ref{lem2.4}, we get  $E(U_{0}) +cM(U_{0})=E(U(t_{1})) +cM(U(t_{1})) \geq V(u(t_{1})) \geq d(c)$ implying that $U_{0}$ is not in $\widetilde{\Sigma }_{\alpha }$, which is a contradiction.
\end{proof}

\setcounter{equation}{0}
\section{Instability of Traveling Waves}
\noindent

We first  compute $d(c)$ and some related quantities.  It follows from  (\ref{Vu}) and  Lemma \ref{lem2.3} that
    \begin{eqnarray}
        d(c)
        &=& V(\phi_{c})  \nonumber \\
        &=& \frac{A}{2}\left\Vert \phi_{c}\right\Vert_{L^{2}}^{2}
            +\frac{B}{2}\left\Vert \phi_{c}^{\prime }\right\Vert_{L^{2}}^{2}-\frac{1}{p+1}\left\Vert \phi_{c}\right\Vert_{L^{p+1}}^{p+1}.
    \end{eqnarray}%
Using the Pohozaev identities, (\ref{Ph1}) and (\ref{Ph2}), in this equation yields
    \begin{equation}
    d(c)=A \left({{p-1}\over{p+3}}\right) \left\Vert \phi_{c}\right\Vert_{L^{2}}^{2}.    \label{ddddd}
    \end{equation}
    We observe from (\ref{solitary}) that $\phi _{c}$ and $\phi _{0}$ are related  through the scaling:
    \begin{displaymath}
    \phi_{c}(x)=A^{\frac{1}{p-1}}\phi_{0}(a^{\frac{1}{2}}A^{\frac{1}{2}}B^{-\frac{1}{2}}x),
    \end{displaymath}
    so that
    \begin{displaymath}
        \left\Vert \phi_{c}\right\Vert_{L^{2}}^{2}
        = a^{-\frac{1}{2}}A^{\frac{5-p}{2p-2}}B^{\frac{1}{2}}\left\Vert \phi_{0}\right\Vert_{L^{2}}^{2}.
    \end{displaymath}
Substituting this into (\ref{ddddd}) we obtain
    \begin{equation}
     d(c)=  a^{-\frac{1}{2}}(1-c^{2})^{\frac{p+3}{2(p-1)}}(a-bc^{2})^{\frac{1}{2}} d(0),  \label{dcc}
   \end{equation}
 where
   \begin{displaymath}
    d(0) =\frac{p-1}{p+3}\left\Vert \phi_{0}\right\Vert_{L^{2}}^{2}>0.
    \end{displaymath}

Our main result is the following theorem showing that  traveling waves with $c^{2}<c_{0}^{2}$ are unstable by blow-up in a finite time.
\begin{theorem}
    Suppose $c^{2}<c_{0}^{2}$ where $c_{0}^{2}$ is given by (\ref{thresh3}), and  $\phi_{c}$ is a traveling wave solution of (\ref{dde0}) with velocity $c$. Let $\Phi_{c}=(\phi_{c}, -c\phi_{c})$ be the corresponding solution of (\ref{sys1})-(\ref{sys2}). There exists initial data $U_{0}$ arbitrarily close to $\Phi_{c}$ in $H^{1}\times H^{1}$ such that the $H^{1}\times H^{1}$ norm of the solution $U(t) =(u(t), w(t))$ of (\ref{sys1})-(\ref{sys3}) blows up in finite time.
\end{theorem}
\begin{proof}
     We  consider the solution $\Phi_{c}=(\phi_{c}, -c\phi _{c}) $ of (\ref{sys1})-(\ref{sys2}), corresponding to the traveling wave solution $\phi_{c}$. For $\lambda >1$, we let
    \begin{eqnarray*}
        h(\lambda)&=& E(\lambda \Phi_{c}) +cM(\lambda \Phi_{c})=V(\lambda \phi_{c})  \\
        &=& \frac{1}{2}\left(A\left\Vert \phi_{c}\right\Vert_{L^{2}}^{2}+B\left\Vert \phi_{c}^{\prime}\right\Vert_{L^{2}}^{2}\right)\lambda ^{2} -\frac{1}{p+1}\left\Vert \phi_{c}\right\Vert_{L^{p+1}}^{p+1}\lambda^{p+1},
    \end{eqnarray*}
    where we have used (\ref{Vu}) and (\ref{V}). The function $h(\lambda)$ has a local maximum at
     \begin{displaymath}
    \lambda_{\max} =\left( { {A\left\Vert \phi_{c}\right\Vert_{L^{2}}^{2}+B\left\Vert \phi_{c}^{\prime}\right\Vert_{L^{2}}^{2}}\over
        {\left\Vert \phi_{c}\right\Vert_{L^{p+1}}^{p+1}}}\right)^{{1}\over {p-1}}.
    \end{displaymath}
    The Pohozaev identity (\ref{Ph1}) implies that $\lambda_{\max}=1$. Then, for $\lambda >1$ ($\lambda$ near 1) we have
    \begin{equation}
    E(\lambda \Phi_{c}) +cM(\lambda \Phi_{c})< V(\phi_{c})=d(c).
    \end{equation}
    As $\lambda^{p+1}>\lambda^{2}$, using (\ref{Ku}) we get
    \begin{eqnarray*}
        \!\!\!\!
        K_{\alpha }(\lambda \phi_{c})
        &=&\lambda^{2}\frac{A}{2}(2\alpha +1)\left\Vert \phi_{c}\right\Vert_{L^{2}}^{2}
            +\lambda^{2}\frac{B}{2}(2\alpha-1)\left\Vert \phi_{c}^{\prime}\right\Vert_{L^{2}}^{2}
            -\lambda^{p+1}(\alpha +\frac{1}{p+1}) \left\Vert \phi_{c}\right\Vert_{L^{p+1}}^{p+1} \\
        \!\!\!\!\!\!\!\!\!\!\!\!\!\!\!\!\!\!\!\!\!\!\!\!\!\!
        &<& \lambda^{2}K_{\alpha}(\phi_{c})=0.
    \end{eqnarray*}%
    The above two results imply that  $\lambda \Phi_{c}\in \widetilde{\Sigma}_{\alpha}$. We now choose a function $v_{0}$ such that
    \begin{displaymath}
        \widehat{v_{0}}(\xi)= \left\{
            \begin{array}{lll}
            \frac{1}{i\xi}\lambda\widehat{\phi_{c}}(\xi)  &  ~~\mbox{for}~~~ \left\vert \xi \right\vert \geq h>0, \\
             0                                           &  ~~\mbox{for}~~~  \left\vert \xi \right\vert <h
             \end{array}
             \right.
    \end{displaymath}
    and set  $U_{0}=((v_{0})_{x}, -c(v_{0})_{x})$. We note that $\left\Vert U_{0}-\Phi_{c}\right\Vert_{H^{1}\times H^{1}}$  can be made arbitrarily small by choosing $\lambda -1$ and $h$ sufficiently small. Thus, by continuity of the functionals, we get that $U_{0}\in \widetilde{\Sigma}_{\alpha}$. By Lemma \ref{lem2.5} it follows that the solution of (\ref{sys1})-(\ref{sys3})  with initial value $U_{0}$ stays in  $ \widetilde{\Sigma}_{\alpha}$ as long as it exists: $U(t) =(u(t), w(t)) \in \widetilde{\Sigma}_{\alpha}$.  Also, using (\ref{momentum}), (\ref{Ph1}), (\ref{Ph2}) and $d(c)= V(\phi_{c})$ we obtain
    \begin{eqnarray*}
        -2cM(\Phi_{c})  &=& -2cM(\phi_{c}, -c\phi_{c})
        =2c^{2}(\left\Vert \phi_{c}\right\Vert_{L^{2}}^{2}+b\left\Vert \phi_{c}^{\prime}\right\Vert_{L^{2}}^{2})  \\
        &=& 2c^{2}\left[ 1+\frac{b(p-1)(1-c^{2})}{(p+3)(a-bc^{2})}\right] \left\Vert \phi_{c}\right\Vert_{L^{2}}^{2} \\
        &=&\frac{2c^{2}}{1-c^{2}}\left[ 1+\frac{b(p-1)(1-c^{2})}{(p+3)(a-bc^{2})}\right] \left(\frac{p+3}{p-1}\right)d(c) .
    \end{eqnarray*}%
    Consequently, for $\lambda >1$, we have
    \begin{displaymath}
    \!\!\!\!\!\!\!\!\!\!\!\!\!\!\!\!\!\!\!\!\!\!
        -2cM(\lambda \Phi_{c}) >-2cM(\Phi_{c}) =\frac{2c^{2}}{1-c^{2}}\left[1+\frac{b(p-1)(1-c^{2})}{(p+3)(a-bc^{2})}\right] \left(\frac{p+3}{p-1}\right)d(c).
    \end{displaymath}
    Again, by continuity, this leads to the following estimate that will be used later:
    \begin{equation}
        -2cM(U_{0}) > \frac{2c^{2}}{1-c^{2}}\left[ 1+\frac{b(p-1)(1-c^{2})}{(p+3)(a-bc^{2})}\right] \left(\frac{p+3}{p-1}\right)d(c) . \label{mmm}
    \end{equation}
   We now define
    \begin{displaymath}
    H(t) =\frac{1}{2}\left(\left\Vert v(t) \right\Vert_{L^{2}}^{2}+b\left\Vert u(t) \right\Vert_{L^{2}}^{2}\right),
    \end{displaymath}
     where $v$ is defined as
    \begin{displaymath}
        v(t) =v_{0}+\int_{0}^{t}w(\tau) d\tau .
    \end{displaymath}
    Note that, due to (\ref{sys1}),   $u=v_{x}$ and $w=$ $v_{t}$.
    We will now show that $H(t)$ blows up in finite time. As in \cite{erbay-NA}, this will ensures that the solution $U(t)$ will blow up in $H^{1}\times H^{1}$  in finite time. To this end we employ Levine's Lemma \cite{levine} and start by estimating $H^{\prime \prime }(t)$. For convenience we suppress the dependencies on $t$ from now on.    Since $v_{t}=w$ and $u_{t}=w_{x}$, using (\ref{sys2}) we get
    \begin{eqnarray}
        && H^{\prime }=\langle v, w\rangle+b \langle u, w_{x}\rangle,  \label{Hfirst}   \\
        && H^{\prime\prime}= \left\Vert w \right\Vert_{L^{2}}^{2}+b\left\Vert w_{x} \right\Vert_{L^{2}}^{2}
           - \left\Vert u \right\Vert_{L^{2}}^{2}-a\left\Vert u_{x} \right\Vert_{L^{2}}^{2}+\left\Vert u \right\Vert_{L^{p+1}}^{p+1}.  \label{Hsecond}
    \end{eqnarray}%
    From the energy conservation we have
    \begin{displaymath}
        \!\!\!\!\!\!\!\!\!\!\!\!\!\!\!\!\!\!\!\!\!\!
        E(U)=\frac{1}{2} \left( \left\Vert w \right\Vert_{L^{2}}^{2} +b   \left\Vert w_{x} \right\Vert_{L^{2}}^{2}
        +\left\Vert u \right\Vert_{L^{2}}^{2}+a\left\Vert u_{x} \right\Vert_{L^{2}}^{2}\right)-\frac{1}{p+1}\left\Vert u \right\Vert_{L^{p+1}}^{p+1}=E(U_{0}).
    \end{displaymath}
    Eliminating $\left\Vert u \right\Vert_{L^{p+1}}^{p+1}$  in (\ref{Hsecond}) we get
    \begin{eqnarray}
        H^{\prime \prime }(t)
        &=&\frac{p+3}{2}\left(\Vert w+cu\Vert_{L^{2}}^{2}+b\Vert w_{x}+cu_{x}\Vert_{L^{2}}^{2}\right)-2cM(U_{0}) \nonumber  \\
        && -(p+1) [E(U_{0})+c M(U_{0})] +J_{c}(u), \label{hprime}
    \end{eqnarray}%
    where
    \begin{equation}
        \!\!\!\!\!\!\!\!\!\!\!\!\!\!\!\!\!\!\!\!\!\!
        J_{c}(u) = {{p-1}\over 2}\left\{ \left[1-c^{2}\left(\frac{p+3}{p-1}\right)\right]\Vert u\Vert_{L^{2}}^{2}
                   +\left[a-bc^{2}\left(\frac{p+3}{p-1}\right)\right]\Vert u_{x}\Vert_{L^{2}}^{2}\right\}.    \label{jc}
    \end{equation}%
     To control $J_{c}(u)$ we first  claim that there are  constants $\alpha > {1\over 2} $  and $C>0$ such that
    \begin{displaymath}
        J_{c}(u)
        = C\left[ V(u)-\frac{1}{\alpha (p+1)+1}K_{\alpha }(u) \right] .
    \end{displaymath}%
    Note that the coefficient of $K_{\alpha }(u) $ is chosen so that the  term  $\left\Vert u \right\Vert_{L^{p+1}}^{p+1}$ disappears.  We then have
    \begin{eqnarray}
    &&  \!\!\!\!\!\!\!\!\!\!\!\!\!\!\!\!\!\!\!\!\!\!
      V(u)-\frac{1}{\alpha (p+1)+1}K_{\alpha }(u) \nonumber   \\
    && \!\!\!\!\!\!\!\!\!
        =\frac{(1-c^{2})}{2}\left(\frac{\alpha (p-1)}{\alpha (p+1)+1}\right)\Vert u\Vert_{L^{2}}^{2}+\frac{a-bc^{2}}{2}\left(\frac{\alpha (p-1)+2} {\alpha(p+1)+1}\right)\Vert u_{x}\Vert_{L^{2}}^{2}  \nonumber\\
    && \!\!\!\!\!\!\!\!\! ={1\over C} \left\{ \frac{p-1}{2}\left[1-c^{2}\left(\frac{p+3}{p-1}\right)\right]\Vert u\Vert_{L^{2}}^{2} \right. \nonumber \\
    && \left. +\frac{(a-bc^{2})}{2(1-c^{2})\alpha }\left[1-c^{2}\left(\frac{p+3}{p-1}\right)\right][\alpha (p-1)+2]\Vert u_{x}\Vert _{L^{2}}^{2}\right\}     \label{ppp}
    \end{eqnarray}%
    where we set
    \begin{equation}
    C=  \frac{[\alpha (p+1)+1]}{(1-c^{2})\alpha}\left[1-c^{2}\left(\frac{p+3}{p-1}\right)\right].  \label{ccc}
    \end{equation}%
    The coefficient of $\Vert u_{x}\Vert _{L^{2}}^{2}$ inside curly brackets in (\ref{ppp}) is the same with that of (\ref{jc}), if we  choose $\alpha$ as follows:
    \begin{equation}
        \alpha =\frac{(a-bc^{2})}{2c^{2}(a-b)}\left[1-c^{2}\left(\frac{p+3}{p-1}\right)\right].  \label{alp}
    \end{equation}%
    Hence, combining (\ref{ccc}) and (\ref{alp}) gives
    \begin{equation}
        C=\frac{(a-bc^{2})\left[1-c^{2}\left(\frac{p+3}{p-1}\right)\right](p+1)+2c^{2}(a-b)}{(1-c^{2})(a-bc^{2})}.  \label{cccnew}
    \end{equation}
     To ensure $\alpha >{1\over 2}$ we must have
    \begin{displaymath}
        (a-bc^{2})\left[1-c^{2}\left(\frac{p+3}{p-1}\right)\right]> c^{2}(a-b).
    \end{displaymath}%
    This can be simplified as follows
    \begin{equation}
       k(c^{2})= b(p+3)c^{4}-2a(p+1)c^{2}+a(p-1)> 0.  \label{baa}
    \end{equation}%
    Since $k(0)>0$ and $k(1)<0$, the function $k(c^{2})$ has only one zero on the interval $(0, 1)$. Then, (\ref{baa}) is satisfied if $c^{2}< c_{0}^{2}$ with
   \begin{eqnarray*}
    c_{0}^{2}&=& \frac{a}{b} \left(\frac{p+1}{p+3}\right) \left[ 1-\left( 1-\frac{b(p+3)(p-1)}{a(p+1)^2} \right)^{1/2}\right] \\
                    &=&  \left(\frac{p-1}{p+1}\right) \left[ 1+\left( 1-\frac{b(p+3)(p-1)}{a(p+1)^2} \right)^{1/2}\right]^{-1} .
    \end{eqnarray*}
    Finally,  it follows from (\ref{cccnew})  that $C>0$ since
    \begin{displaymath}
     c_{0}^{2}\leq  \frac{p-1}{p+3}.
    \end{displaymath}
    We next claim that   $J_{c}(u) \geq Cd(c)$. Since $u\in \widetilde{\Sigma}_{\alpha}$, we have $K_{\alpha}(u) <0$. We can then find $0< \gamma <1$ so that  $K_{\alpha}(\gamma u) =0$. By Lemma \ref{lem2.3}, this implies that $V(\gamma u) \geq d(c)$. But then
    \begin{eqnarray}
    J_{c}(u)& > & \gamma^{2}J_{c}(u) =J_{c}(\gamma u)=C\left[ V(\gamma u)-\frac{1}{\alpha (p+1)+1}K_{\alpha }(\gamma u)\right] \nonumber \\
        & = & CV(\gamma u) \geq Cd(c) \label{jcu}
    \end{eqnarray}
    which proves our claim. We are now in the position of putting all the above calculations together to estimate $H^{\prime\prime}$.    Writing $E(U_{0})+cM(U_{0}) =d(c) -\delta $ with $\delta >0$ in (\ref{hprime}) and using (\ref{mmm}), (\ref{jcu}), we get
    \begin{eqnarray*}
        H^{\prime \prime }
        & \geq &\frac{p+3}{2}\left( \Vert w+cu\Vert_{L^{2}}^{2}+b\Vert w_{x}+cu_{x}\Vert_{L^{2}}^{2}\right) -(p+1)d(c) +(p+1) \delta  \\
        &~& + \frac{2c^{2}}{1-c^{2}}\left[ 1+\frac{b(p-1)(1-c^{2})}{(p+3)(a-bc^{2})}\right] \left(\frac{p+3}{p-1}\right)d(c) +C d(c)  \\
        &=& \frac{p+3}{2}\left( \Vert w+cu\Vert _{L^{2}}^{2}+b\Vert w_{x}+cu_{x}\Vert_{L^{2}}^{2}\right)  +( p+1) \delta +\sigma d(c)
    \end{eqnarray*}%
    where
    \begin{eqnarray*}
        \sigma  &=&-( p+1) +\frac{2c^{2}}{1-c^{2}}\left[ 1+\frac{b(p-1)(1-c^{2})}{(p+3)(a-bc^{2})}\right] \left(\frac{p+3}{p-1}\right) \\
            &&+\frac{(a-bc^{2})\left[1-c^{2}\left(\frac{p+3}{p-1}\right)\right](p+1)+2c^{2}(a-b)}{(1-c^{2})(a-bc^{2})}.
    \end{eqnarray*}%
    A direct calculation shows that  $\sigma $ is zero to yield
    \begin{displaymath}
        H^{\prime \prime }  \geq \frac{p+3}{2}\left( \Vert w+cu\Vert_{L^{2}}^{2}+b\Vert w_{x}+cu_{x}\Vert _{L^{2}}^{2}\right) +(p+1)\delta.
    \end{displaymath}
    So, $H^{\prime \prime }\left( t\right) >\left( p+1\right) \delta $  which in turn implies that $H^{\prime }\left( t_{0}\right) >0$ for some $ t_{0}>0$.     Since  $u=v_{x}$ we have $ \langle v, u \rangle = \langle u, u_{x} \rangle=0$. Then from (\ref{Hfirst})
    \begin{displaymath}
     H^{\prime }
        = \langle v, w+cu \rangle + b \langle u, w_{x}+cu_{x} \rangle .
    \end{displaymath}
    Thus
      \begin{displaymath}
     (H^{\prime })^{2}
        \leq \left(\Vert v\Vert_{L^2}^{2}+ b \Vert u \Vert_{L^2}^{2}\right)   \left(\Vert w+cu\Vert_{L^2}^{2}+ b \Vert w_{x}+cu_{x} \Vert_{L^2}^{2}\right).
    \end{displaymath}
    Finally,  we have
   \begin{displaymath}
    H H^{\prime \prime }-\frac{p+3}{4}\left( H^{\prime } \right)^{2}
        \geq (p+1) H\delta \geq 0.
    \end{displaymath}
    By Levine's Lemma \cite{levine} this shows that $H(t)$ blows up in finite time. This completes the proof.
\end{proof}
\vspace*{20pt}

\setcounter{equation}{0}
\section{Stability Regions for Traveling Waves}
\noindent

In this section we investigate both analytically and numerically the dependence of stability  regions of traveling wave solutions of the double dispersion equation on the parameters $a$, $b$ and $p$. To be precise, by the stability region we mean the set of wave velocities $c$ for which  the traveling wave solutions of (\ref{dde0}) are orbitally stable.  Recall that a traveling wave $\phi_{c}$ is said to be orbitally stable if any solution $U(t)$ with initial data sufficiently close to the traveling wave stays close, at any later time,   to some translate of $\phi_{c}$. It is a well-known phenomena in nonlinear wave theory that orbital stability occurs for all values of $c$ for which a scalar function $d(c)$ is convex \cite{strauss, souganidis, esfahani}. For  the general class given by (\ref{class0}), this  was proved explicitly in \cite{erbay-JMAA}. To apply the convexity criterion to the double dispersion equation, we first rewrite the function $d(c)$ given in (\ref{dcc}) as
\begin{equation}
      d(c)
       = d(0)(1-c^{2})^{\frac{p+3}{2(p-1)}}(1-\mu c^{2})^{\frac{1}{2}},  ~~~~~ \mu={b\over a}. \label{dccy}
\end{equation}%
As $0\leq b < a$, we will consider $0\leq \mu < 1 $.  A direct computation of $d^{\prime\prime}(c)$ gives
\begin{equation}
    \!\!\!\!\!\!\!\!\!\!\!\!\!\!\!\!\!\!\!\!\!\!
   d^{\prime\prime}(c)=d(0)(p-1)^{-2}(1-c^{2})^{{7-3p}\over {2(p-1)}}(1-\mu c^{2})^{-3/2}(Pc^{6}-Qc^{4}+Rc^{2}-S)
\end{equation}
with
\begin{eqnarray*}
    P &=&   2(p+3)(p+1)\mu^{2},     \\
    Q &=&   3( p+3)(p-1)\mu^{2}+(3p^{2}+10p+19) \mu,       \\
    R &=&   2((3p+5)(p-1)\mu+2(p+3)),     \\
    S &=&   (p-1)^{2}\mu+(p-1)(p+3).
\end{eqnarray*}
Hence the sign of $d^{\prime\prime}(c)$ is determined by the sign of the polynomial
\begin{equation}
    G(z, p, \mu)=Pz^{3}-Qz^{2}+Rz-S.
\end{equation}
Recalling that  traveling waves exist for $c^{2}<1$, we see that the stability regions are the set of all wave velocities $c$ for which $c^{2}<1$ and $G(c^{2},p,\mu)>0$. So the problem   reduces to the problem of finding  real roots of $G(z,p,\mu) $ on the interval $ (0,1)$. The remainder of this section focuses on  analyzing how the parameters $p$ and $\mu$ affect the locations of the roots and, consequently, the stability regions. We first restrict our attention to the exploration of locations and number of the roots in $(0,1)$ and then focus on formulating  explicit stability conditions in terms of $c$ in the last part of this section.

First we observe that the coefficients $P$, $Q$, $R$ and $S$ are all positive, so all real roots of $G(z,p,\mu )$ must be positive and $G(0,p,\mu)<0$.  As $P$, $Q$, $R$, $S$ are  continuous in the parameters $p$ and $\mu $, the three (possibly complex) roots $z(p, \mu )$ of the cubic polynomial $G(z,p,\mu )$ depend continuously on  $p$ and $\mu $ for $p>1$ and $\mu >0$.

It will be useful to consider what happens at $z=1$. Computation gives%
\begin{equation}
G(1,p,\mu )=(\mu -1)^{2}(p+3)(5-p).  \label{z=1}
\end{equation}%
Hence, for $\mu<1$, $G(1, p, \mu)>0$ when $p<5$ but $G(1, p, \mu)<0$ when $p>5$. Since $G(0, p, \mu)<0$, the number of distinct  roots of $ G(z, p, \mu)$ in the interval $(0,1)$ must be  {\it (i)} one or three when $p<5$ and {\it (ii)}   zero or two when $p>5$.

Equation (\ref{z=1}) shows that when $p=5$  we have the root  $z_{1}(5, \mu ) =1$. This will allow us to determine completely the case $p=5$.
Factoring  $G(z,5,\mu)$, we get
\begin{equation}
G(z,5,\mu )=16(z-1)(6\mu ^{2}z^{2}-9\mu  z +\mu +2),
\end{equation}%
which yields  the other two distinct roots
\begin{equation}
z_{\pm }(5,\mu ) =\frac{1}{12\mu }( 9\pm \sqrt{33-24\mu }) .
\end{equation}
We now try to locate the roots $z_{\pm }(5,\mu )$ of $G(z, 5, \mu )$ in $(0,1)$. Since $0\leq \mu < 1$, the roots $z_{\pm }(5,\mu )$ are real and, in consequence, there are three real roots. First note that $ \mu < 1$ implies  $ z_{+}(5,\mu ) \geq \frac{1}{\mu }> 1$. On the other hand,  an easy computation shows that $z_{-}(5, \mu ) <1$  if and only if  $ \frac{1}{3}<\mu < 1$. Summing up, we have: $G(z, 5,\mu ) $ has one root $z_{-}(5, \mu )$ in $(0, 1) $  for $\frac{1}{3}<\mu <1$ but no root in  $( 0,1) $ for $ 0 \leq \mu \leq \frac{1}{3}$.

Next we want to use continuity of the roots with respect to the parameters to understand what happens when $p$ is near 5 and $\mu$ is fixed. We first decrease $p$ slightly from $5$. Recall that $G(z,p,\mu)$ must have exactly one root or three roots in $(0,1)$ for $p<5$. Since $z_{+}(5, \mu )>1$ for $\mu <1$, we cannot have the case of three roots. Therefore, for $p$ slightly smaller than 5 and $0\leq \mu < 1$, $G(z,p,\mu)$ will have exactly one root in $(0,1)$. To determine what happens  when $p$ increases slightly from 5, we have to consider two cases: $\mu <{1\over 3}$ and $\mu >{1\over 3}$. When $\mu <{1\over 3}$,  $G(z,5,\mu)$ has two roots $z_{-}(5,\mu)$ and $z_{+}(5,\mu)$ in $(1, \infty)$. For $p$ sufficiently close to $5$, none of these two roots can move into $(0,1)$. Recalling that for $p>5$, $G(z,p,\mu)$  must have zero or two roots in $(0,1)$, and noting that we have just eliminated the possibility of two roots when  $p$ is slightly greater than 5 and $\mu <{1\over 3}$, we conclude that $G(z,p,\mu)$ has no root in $(0,1)$. Finally, consider the case $\mu >{1\over 3}$ with $p$ slightly larger than 5. Since $z_{-}(5,\mu)< z_{1}(5,\mu)=1$, the only possibility  is that the root $z_{1}(5,\mu)$ moves to the left, yielding exactly two roots in $(0,1)$ for $G(z,p,\mu)$.

\begin{figure}[ht]
    \includegraphics[width=200pt]{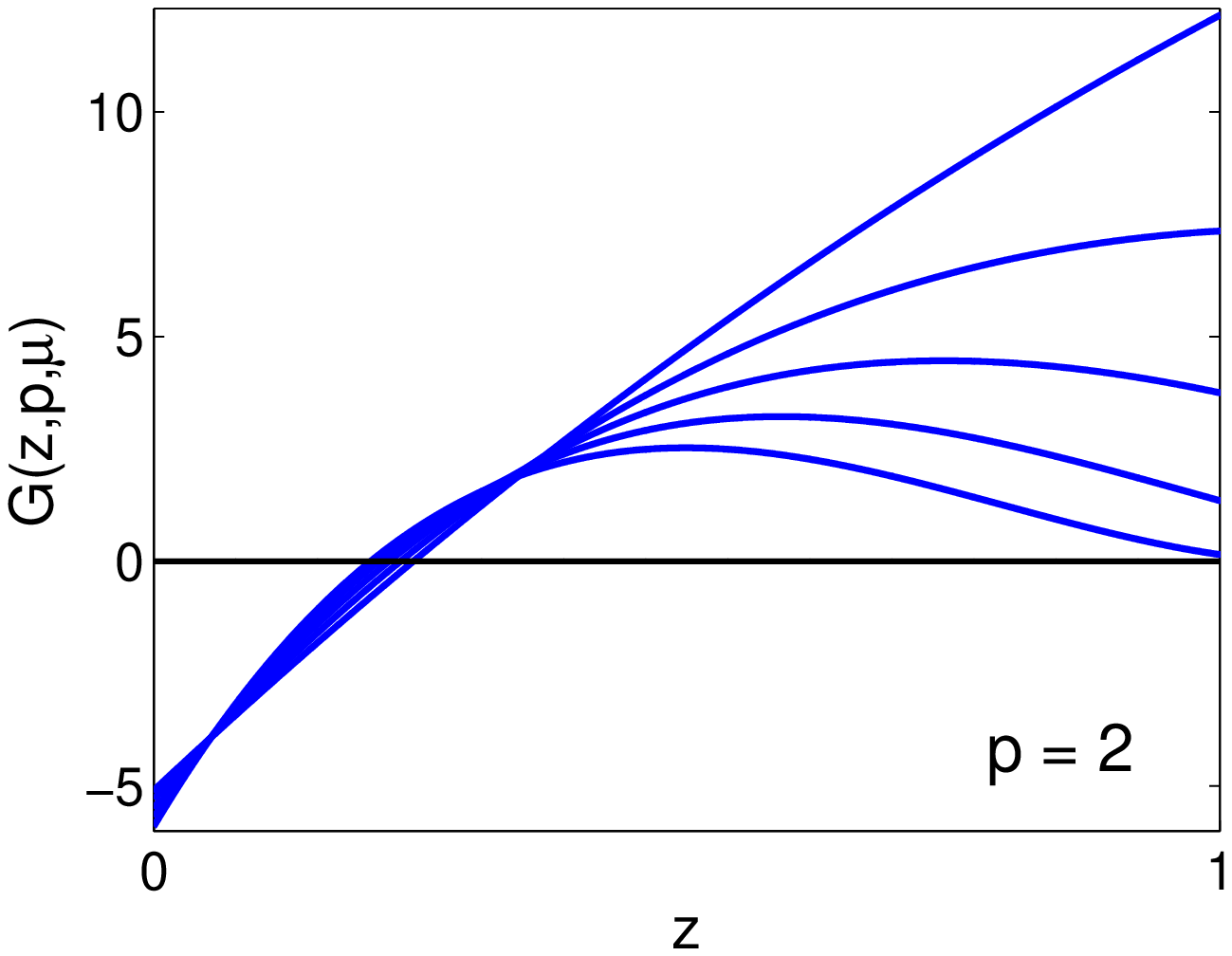}%
    \hspace{10pt}%
    \includegraphics[width=200pt]{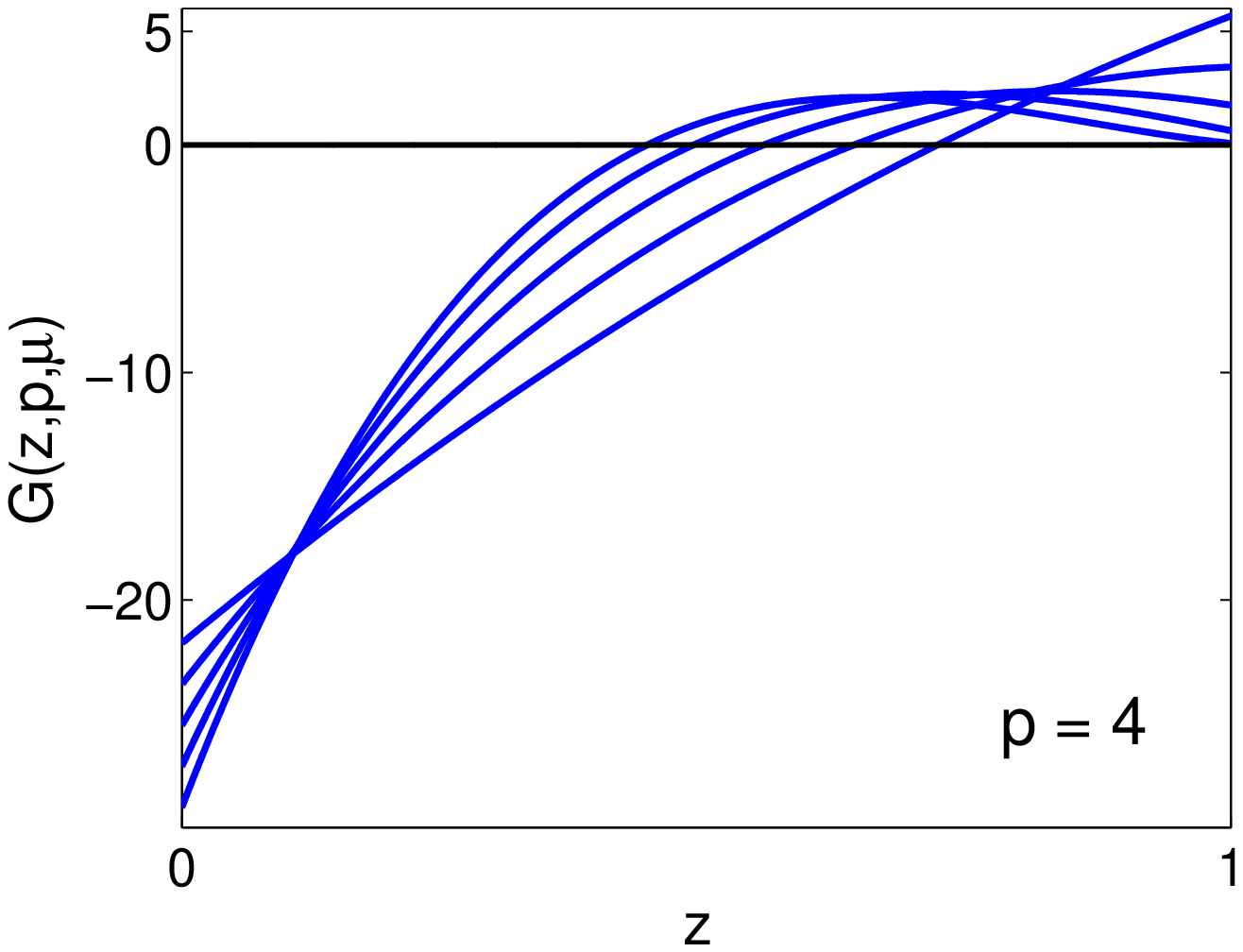}
    \caption{Variation of the function $G(z,p,\mu)$ with $z$ on the interval $[0,1]$ for (a) $p=2$, (b) $p=4$ and for $\mu=0.1,0.3,0.5,0.7,0.9$ (from top to bottom  at the right end-point).}
\end{figure}

Summing up what we know about the total number of roots   on the interval $(0, 1)$,   we have:
\begin{itemize}
\item For any $0\leq \mu <1$,   $G(z,p,\mu)$ has only one root in $(0,1)$ when $p<5$ and $p$ near 5.
\item  For $\mu < \frac{1}{3}$,   $G(z,p,\mu)$ has no root in $(0,1)$ when $p>5$ and $p$ near 5.
\item  For $ \mu > \frac{1}{3}$,  $G(z,p,\mu)$ has two roots in $(0,1)$ when $p>5$ and $p$ near 5.
\end{itemize}

\begin{figure}[ht]
    \includegraphics[width=200pt]{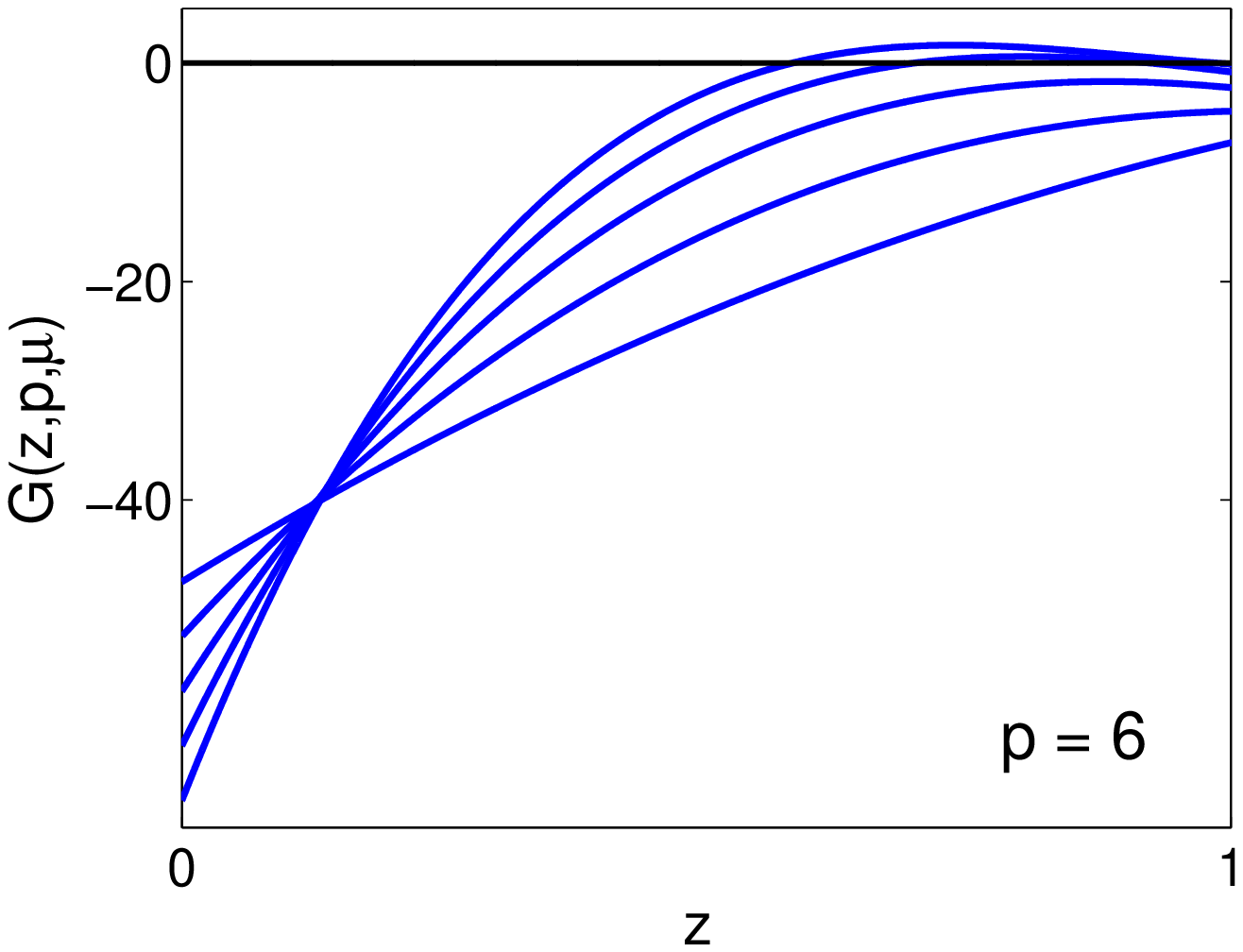}%
    \hspace{10pt}%
    \includegraphics[width=200pt]{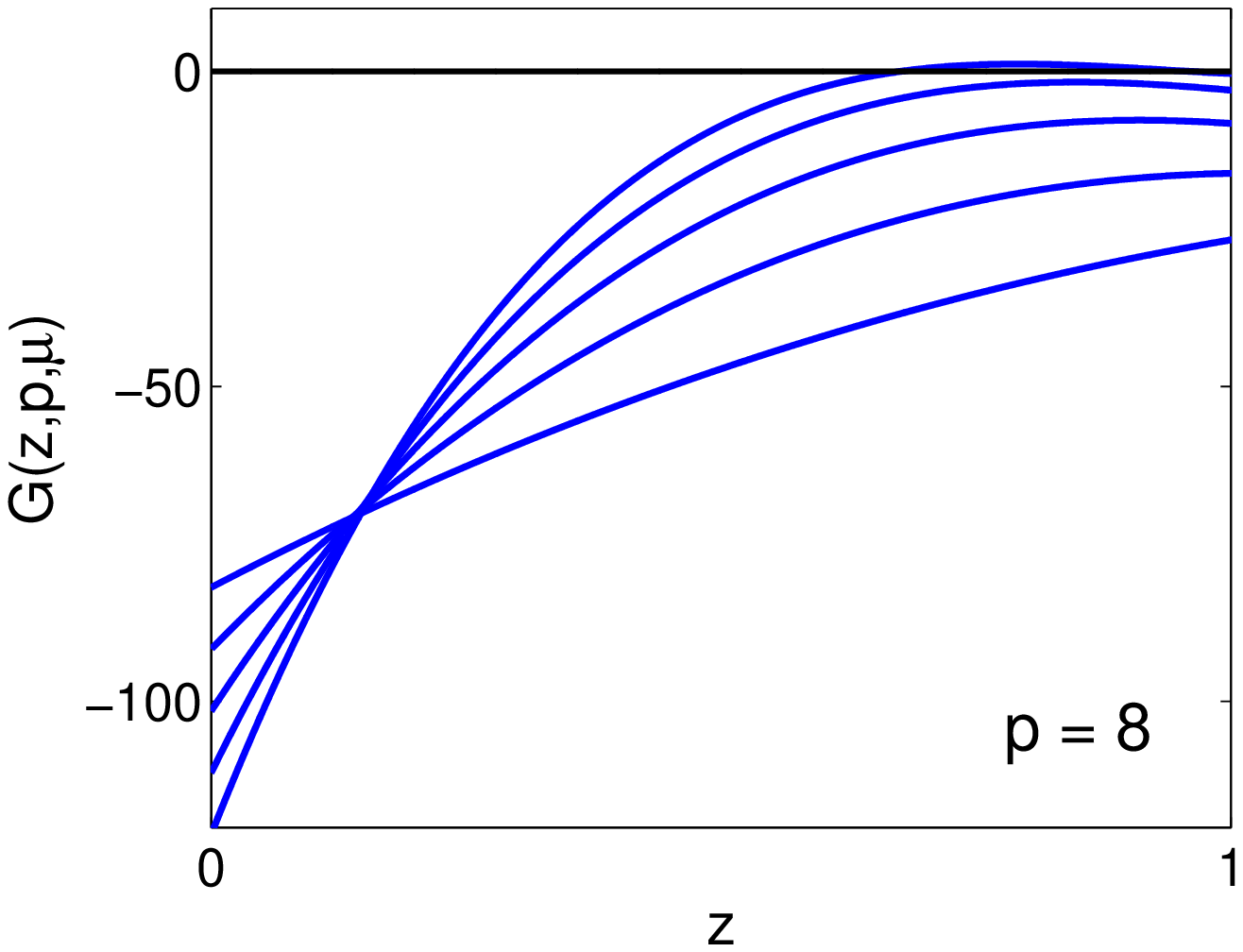}
    \caption{Variation of the function $G(z,p,\mu)$ with $z$ on the interval $[0,1]$ for (a) $p=6$, (b) $p=8$ and for $\mu=0.1,0.3,0.5,0.7,0.9$ (from bottom to top  at the right end-point).}
\end{figure}

For general values of $p$ and $\mu$ we now provide numerical evidence to suggest that exactly the same behavior is observed for other parameter values. In Figures 1 and 2 we present the graph of $G(z,p,\mu)$ as a function of $z$ on $[0,1]$ for $p=2, 4$ and $p=6, 8$, respectively. In each figure, the curves correspond to   the following five cases: $\mu=0.1, 0.3, 0.5, 0.7, 0.9$.      The curves  are identified from (\ref{z=1}) by observing that $G(1, p, \mu)$ is decreasing in $\mu$ for $p<5$ but increasing in $\mu$ for $p>5$. That is, at the right end-point, the curves  correspond to $\mu=0.1, 0.3, 0.5, 0.7, 0.9$  from top to bottom for $p<5$  but from bottom to top for $p>5$, respectively. We see from the figures that the itemized conclusions of the previous paragraph about the number of roots  of $G(z,p,\mu)$  for $p$ near 5 are exactly valid for all values of $p$ and $\mu$ with a critical value $\mu_{p}$ replacing the value $1/3$. Motivated by this fact,  we will make the following claim about the number of roots of $G(z,p, \mu)$ in $(0,1)$:
\begin{itemize}
\item For $p<5$  and $0\leq \mu < 1$,  $G(z,p, \mu)$ has only one root $z_1(p,\mu)$ in $(0,1)$.
\item For $p>5$  there is a critical value $\mu_p \in (0,1)$ so that  $G(z,p, \mu)$ has no root in $(0,1)$ for $0 \leq \mu < \mu_p$ but it has two roots  $z_1(p,\mu), z_2(p,\mu)$ in $(0,1)$ for  $\mu_p < \mu <1$.
\end{itemize}
We note that the above claim contains the case of the Boussinesq equation ($\mu=0$), where $G(z,p,0) $ has exactly one root $z_1(p,0)=\frac{p-1}{4}$.
This root is in  $(0,1)$ if and only if $p<5$. Another limiting case where $\mu=1$ was analysed  in \cite{erbay-JMAA}. In this case we have
\begin{equation}
G(z,p,1) = 2(p+1) (p+3) \left(z-\frac{p-1}{p+3}\right) (z-1)^{2},
\end{equation}%
which implies that  $G(z,p,1)$ has only one root $z= {{p-1}\over{p+3}}$ in $(0,1)$.  Note that this can be considered as a limiting case of the claim above with $z_2(p,1)=1$.

\begin{figure}[ht]
    \centering
    \includegraphics[width=300pt]{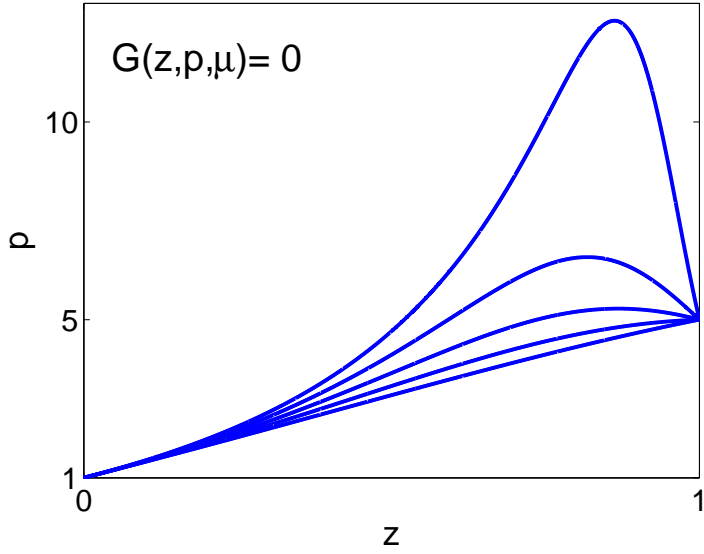}%
    \caption{Variation of $p$ with $z$ on the interval $[0,1]$ when $G(z,p,\mu)=0$ for  $\mu=0.1,0.3,0.5,0.7,0.9$ (from bottom to top).}
\end{figure}

The information collected for the locations of the roots of $G(z,p,\mu)$ allows us to determine the stability regions as follows:
\begin{itemize}
\item  If   $G(z,p,\mu)$ has no root in $(0,1)$, then the stability region is empty.
\item  If   $G(z,p,\mu)$ has one root $z_1(p,\mu)$ in $(0,1)$, then the stability region is the set of wave velocities satisfying $z_1(p,\mu)< c^2 <1$.
\item  If   $G(z,p,\mu)$ has two  roots $z_1(p,\mu)< z_2(p,\mu) $ in $(0,1)$, then the stability region is the set of wave velocities satisfying
$z_1(p,\mu)< c^2 < z_2(p,\mu)$.
\end{itemize}

\begin{figure}[ht]
    \centering
    \includegraphics[width=300pt]{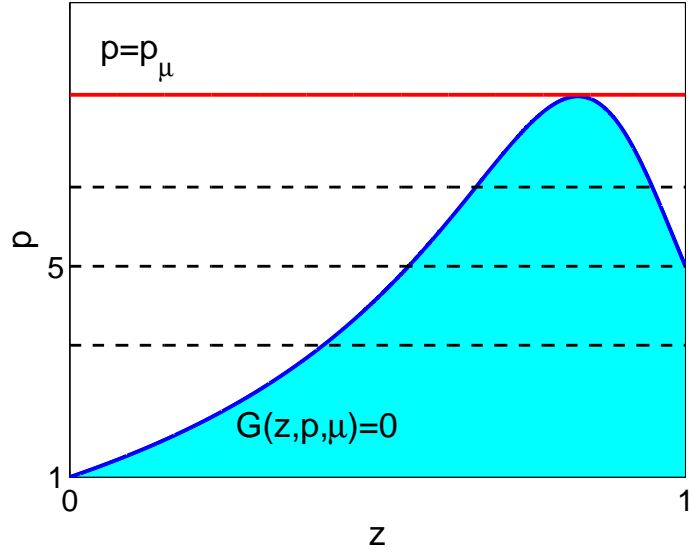}%
    \caption{Schematic diagram of the stability region (shaded region) for a fixed $\mu$.}
\end{figure}

To illustrate the roots of  $G(z, p, \mu) $ and the corresponding stability regions, we have also graphed the set $G(z, p, \mu) = 0$  in the $zp-$plane for certain fixed values of $\mu$ in Fig. 3.  The curves are ordered from bottom to top: the bottom one is the set $G(z, p, 0) = 0$, and the curves move up as $\mu$ increases.  The curves show the location of the real roots of $G(z,p,\mu)$ for the corresponding values of $\mu$. Namely, a point $(z^{*}, p^{*})$ on the curve corresponds to the root $z^{*}$ of $G(z,p^{*},\mu)$. Conforming with our conjecture about the roots, the graph indicates that: {\it (i)} when $p < 5$ there is exactly one root $z_{1}(p, \mu)$ in $(0, 1)$ which decreases as $\mu$ increases; {\it (ii)} when $p > 5$, there is some $\mu_{p}$ such that for $\mu < \mu_{p}$ there is no root in $(0, 1)$ whereas for $\mu > \mu_{p}$ there are two roots $z_{1}(p, \mu) < z_{2}(p, \mu)$ in $(0, 1)$. Moreover, $z_{1}(p, \mu)$ is decreasing in $\mu$, while $z_{2}(p, \mu)$ increases and approaches 1 as $\mu$ approaches 1.  For a fixed $\mu_{0}$, the orbital stability interval is obtained by intersecting the line $p = p_{0}$ with the set $G(z, p, \mu_{0})=0$, this set in turn is either empty or an interval for $c^2$ of the form $(z_{1}(p_0, \mu_0), 1)$ or $(z_{1}(p_0, \mu_0),  z_{2}(p_0, \mu_0) )$. Fig. 3 also indicates that the critical value $\mu_{p}$ increases with $p$. This means that we can as well fix $\mu$ and vary $p$. Then there is a critical value $p=p_{\mu}$ so that  when $p\geq p_{\mu}$ the stability regions are empty. To illustrate this, in Fig. 4 we take a single curve $G(z,p,\mu)=0$ with fixed $\mu$ and several horizontal lines corresponding to different values of $p$. We observe transitions between different types of stability regions as $p$ varies for a fixed $\mu$.  Fig. 4 also gives the critical value $p_{\mu}$. The shaded region in Fig. 4, that is,  the area between  the curve $G(z,p,\mu)=0$ and the line $p=1$,  corresponds to the stability regions of the problem for varying $p$.

To conclude, our analysis in this section leads to the following observation:  {\em   Traveling wave solutions of  the double dispersion equation (\ref{dde0}) are orbitally stable in each of the following three cases;
\begin{eqnarray}
  && (A)~~  p<5  ~~~\mbox{and}~~~ z_{1}(p,\mu)<c^{2}<1, \label{conA} \\
  &&  (B)~~ p=5,   ~~~~{1\over 3}<\mu<1 ~~~\mbox{and}~~~~  \frac{1}{12\mu }( 9- \sqrt{33-24\mu }) <c^{2}<1, \label{conB}\\
  &&  (C)~~ p>5, ~~~~ \mu_{p}<\mu<1 ~~~\mbox{and}~~~ z_{1}(p,\mu)<c^{2}< z_{2}(p,\mu)<1. \label{conC}
\end{eqnarray}
Moreover, for a fixed $p$, as $\mu$ increases,  the stability interval gets larger. Also, for $p>5$, the critical value $\mu_{p}$ increases as $p$ increases. }
\vspace*{30pt}

%\section*{References}

\end{document}